\documentclass[oneside,10pt]{article} 

\usepackage[b5paper]{geometry}	    

\usepackage{amsfonts,amsmath,latexsym,amssymb} 
\usepackage{amsthm}                
\usepackage{mathrsfs,upref}        
\usepackage{mathptmx}		    
\usepackage{my_jmi}	            

\usepackage{amsrefs}

\newtheorem{theorem}{Theorem}           
\newtheorem{lemma}[theorem]{Lemma}               
\newtheorem{corollary}[theorem]{Corollary}

\theoremstyle{definition}

\newtheorem{definition}[theorem]{Definition}

\newtheorem{remark}[theorem]{Remark}

\begin{document}
	
\title{Generalized weighted Ostrowski and Ostrowski-Gr\"uss 
type inequalities on time scales via a parameter function}

\author{Seth Kermausuor, Eze R. Nwaeze and Delfim F. M. Torres}

\address{Seth Kermausuor, 
Department of Mathematics and Statistics, 
Auburn University, Auburn, AL 36849, USA\\
\email{skk0002@auburn.edu}}

\address{Eze R. Nwaeze, 
Department of Mathematics,
Tuskegee University, 
Tuskegee, AL 36088, USA\\
\email{enwaeze@mytu.tuskegee.edu}}

\address{Delfim F. M. Torres, 
CIDMA, Department of Mathematics, 
University of Aveiro, 3810-193 Aveiro, Portugal\\
\email{delfim@ua.pt}}

\CorrespondingAuthor{Delfim F. M. Torres}

\date{Submitted 23.12.2016; Revised 21.06.2017; Accepted 27.06.2017}


\keywords{Ostrowski's inequality; 
Ostrowski--Gr\"uss inequality; 
parameter function; 
time scales}

\subjclass{26D10, 26D15, 26E70}

\thanks{This research was partially supported 
by Portuguese funds through CIDMA and FCT, 
within project UID/MAT/04106/2013} 


\begin{abstract}
We prove generalized weighted Ostrowski and Ostrowski--Gr\"uss type inequalities 
on time scales via a parameter function. In particular, our result extends a result 
of Dragomir and Barnett. Furthermore, we apply our results to the continuous, 
discrete, and quantum cases, to obtain some interesting new inequalities. 
\end{abstract}

\maketitle


\section{Introduction}

In order to estimate the absolute deviation of a differentiable function 
from its integral mean, Dragomir and Barnett \cite{DB} obtained in 1999
the following Ostrowski type inequality.

\begin{theorem}[See \cite{DB}]
\label{DB}
Let $f:[a,b]\to\mathbb{R}$ be continuous on $[a,b]$ and twice differentiable 
on $(a,b)$ with second derivative $f'':(a,b)\to\mathbb{R}$. Then,
\begin{align*}
\Bigg|f(x)-\frac{f(b)-f(a)}{b-a}\Big(x-\frac{a+b}{2}\Big)
&-\frac{1}{b-a}\int_a^af(t)dt\Bigg|\\
&\leq\frac{M}{2}\Bigg\{\Bigg[\frac{(x-\frac{a+b}{2})^2}{(b-a)^2}
+\frac{1}{4}\Bigg]^2+\frac{1}{12}\Bigg\}(b-a)^2 
\end{align*}
for all $x\in[a,b]$, where $\displaystyle M=\sup_{a<t<b}|f''(t)|<\infty$.
\end{theorem}

By introducing a parameter, Liu \cite{LJ} established in 2010
the following perturbed weighted generalized 
three-point integral inequality with bounded derivative.

\begin{theorem}[See \cite{LJ}]
Let $0\leq k\leq 1$ and $f:[a,b]\to\mathbb{R}$ be a differentiable mapping. 
Assume there exists a constant $\gamma\in\mathbb{R}$ such that 
$\gamma\leq f'(x)$ for $x\in[a,b]$, $g:[a,b]\to[0,\infty)$ 
is continuous and positive on $(a,b)$, and let $h:[a,b]\to\mathbb{R}$ 
be differentiable such that $h'(t)=g(t)$ on $[a,b]$. Then,
\begin{align*}
&\Bigg|\Bigg\{(1-k)f(x)-k\bigg[\frac{\int_a^xg(t)dt}{\int_a^bg(t)dt}f(a)
+\frac{\int_x^bg(t)dt}{\int_a^bg(t)dt}f(b)\bigg]\Bigg\}\int_a^bg(t)dt\\
&\quad -\gamma\Bigg\{(1-k)\bigg[(h(b)-h(a))\bigg(x-\frac{a+b}{2}\bigg)\\
&\quad +(b-a)\bigg(h(x)-\frac{h(a)+h(b)}{2}\bigg)\bigg]
\int_a^b(h(t)-h(x))dt\Bigg\}-\int_a^bf(t)g(t)dt\Bigg|\\
&\leq \begin{cases}
\displaystyle(1-k)\bigg[\frac{1}{2}\int_a^bg(t)dt+\Big|h(x)
-\frac{h(a)+h(b)}{2}\Big|\bigg](S-\gamma)(b-a), & k\in[0,\frac{1}{2}]\\[0.3cm]
\displaystyle k\bigg[\frac{1}{2}\int_a^bg(t)dt+\Big|h(x)
-\frac{h(a)+h(b)}{2}\Big|\bigg](S-\gamma)(b-a), & k\in(\frac{1}{2},1]
\end{cases}
\end{align*}
for all $x\in[a,b]$. 
\end{theorem}

In order to unify the continuous and discrete calculus in a consistent manner, 
Hilger introduced in 1988 the theory of time scales \cites{BookTS:2001,Hilger}. 
Since the advent of this calculus, many researchers have been able to extend known 
classical integral inequalities to time scales. We refer the interested reader 
to papers \cites{BM1,BM2,Liu1,Liu2,Liu3,MR2816120,Nwaeze,Ozkan1,Ozkan2,XuFang}, 
books \cites{MR3307947,MR3445007,MR3526108}, and references therein.
In 2014, Liu et al. \cite{Liu4} obtained the following inequality on time scales.

\begin{theorem}[See \cite{Liu4}]
\label{introthm4}
Let $0\leq k\leq 1$, $g:[a,b]\to[0,\infty)$ be $rd$-continuous and positive, 
and $h:[a,b]\to\mathbb{R}$ be differentiable such that $h^\Delta(t)=g(t)$ on $[a,b]$. 
Let $a, b, t, x\in\mathbb{T}$, $a<b$, and $f:[a,b]\to\mathbb{R}$ be twice differentiable. 
Then, 
\begin{align*}
&\Bigg|(1-k)^2f(x)-\frac{1}{(\int_a^bg(t)\Delta t)^2}\bigg(\int_a^bS(x,t)\Delta t\bigg)\bigg(\int_a^bg(t)f^\Delta(\sigma(t))\Delta t\bigg)\\
&\quad +\frac{k}{(\int_a^bg(t)\Delta t)^2}\int_a^bS(x,t)\bigg(f^\Delta(a)\int_a^tg(s)\Delta s+f^\Delta(b)\int_t^bg(s)\Delta s\bigg)\Delta t\\
&\quad +\frac{k(1-k)}{(\int_a^bg(t)\Delta t)}\bigg(f(a)\int_a^xg(t)\Delta t+f(b)\int_x^bg(t)\Delta t\bigg)\\
&\quad -\frac{1-k}{\int_a^bg(t)\Delta t}\int_a^bg(t)f(\sigma(t))\Delta t\Bigg|\\
&\leq \frac{M}{(\int_a^bg(t)\Delta t)^2}\int_a^b\int_a^b|S(x,t)||S(t,s)|\Delta s\Delta t
\end{align*}
for all $x\in[a,b]$, where $\displaystyle M=\sup_{a<t<b}|f^{\Delta\Delta}(t)|<\infty$ and 
\begin{equation*}
S(x, t)=
\begin{cases}
h(t)-((1-k)h(a)+kh(x)),\quad  a\leq t<x,\\
h(x)-(kh(x)+(1-k)h(b)), \quad x\leq t\leq b.
\end{cases}
\end{equation*}
\end{theorem}

Recently, in 2016, by using a different weighted Peano kernel, 
Nwaeze obtained in \cite{Nwaeze} the following 
weighted Ostrowski type inequality.

\begin{theorem}[See \cite{Nwaeze}]
\label{thm:Nw:16}
Let $\nu:[a, b]\rightarrow [0,\infty)$ be $rd$-continuous and positive 
and $w:[a, b]\rightarrow \mathbb{R}$ be differentiable such that 
$w^{\Delta}(t)=\nu(t)$ on $[a, b]$. Suppose also that $a, b, s, t\in \mathbb{T}$, 
$a<b$, $f:[a, b]\rightarrow \mathbb{R}$ is differentiable, 
and $\psi$ is a function of $[0, 1]$ into $[0, 1]$. Then, 
\begin{align*}
\Bigg|\left[\dfrac{1+\psi(1-\lambda)-\psi(\lambda)}{2}f(t) 
+ \dfrac{\psi(\lambda)f(a)+\left(1-\psi(1-\lambda)\right)f(b)}{2}\right]\int_a^b \nu(t)\Delta t\\
-\int_a^b \nu(s)f(\sigma(s))\Delta s\Bigg|\leq M\int_a^b |K(s,t)|\Delta s,
\end{align*}
where
\begin{equation}
\label{WE}
K(s, t)=
\begin{cases}
w(s)-\left(w(a)+\psi(\lambda)\frac{w(b)-w(a)}{2}\right), \quad s\in[a, t),\\
w(s)-\left(w(a)+(1+\psi(1-\lambda))\frac{w(b)-w(a)}{2}\right), \quad s\in[t, b],
\end{cases}
\end{equation}
and $M=\sup\limits_{a<t<b}|f^{\Delta}(t)|<\infty$.
\end{theorem}

Inspired by the ideas employed in \cites{Liu4,Nwaeze},
here we obtain generalized Ostrowski and Ostrowski--Gr\"uss inequalities 
on time scales via a parameter function. Our results are different 
from the ones given in \cite{Liu4} since we are using the generalized 
weighted Peano kernel \eqref{WE}. Furthermore, we apply our result 
to the continuous, discrete, and quantum cases, to obtain some interesting new inequalities. 
More corollaries are also obtained by considering different parameter and weight functions. 
In particular, we generalize and extend Theorem~\ref{DB} to time scales (see Remark~\ref{remDB}).

The paper is organized as follows. In Section~\ref{sec2},
we provide the reader with essentials on the calculus on time scales.  
Our results are then stated and proved in Section~\ref{sec3}.


\section{Preliminaries on time scales}
\label{sec2}

In this section, we briefly recall the theory of time scales. For further
details and proofs we refer the reader to Hilger's original work \cite{Hilger} 
and to the books \cites{BookTS:2001,BookTS:2003}.

\begin{definition}
A time scale is an arbitrary nonempty closed subset of the real numbers 
$\mathbb{R}$.
\end{definition}

Throughout this work, we assume $\mathbb{T}$ to be a time scale with 
the topology that is inherited from the standard topology on $\mathbb{R}$. 
It is also assumed throughout that in $\mathbb{T}$ the interval 
$\left[ a,b\right]$ means the set $\left\{ t\in \mathbb{T}\text{: }a\leq t
\leq b\right\}$ for points $a<b$ in $\mathbb{T}$. Since a time scale may not
be connected, we need the following concept of jump operators.

\begin{definition}
The forward and backward jump operators 
$\sigma ,\rho :\mathbb{T}\rightarrow \mathbb{T}$ 
are defined by
$\sigma (t)=\inf \left\{ s\in \mathbb{T} : s>t\right\}$ 
and $\rho (t)=\sup \left\{ s\in \mathbb{T} : s<t\right\}$,
respectively.
\end{definition}

The jump operators $\sigma $ and $\rho $ allow the classification 
of points in $\mathbb{T}$ as follows.

\begin{definition}
If $\sigma (t)>t$, then $t$ is right-scattered, 
while if $\rho(t)<t$, then we say that $t$ is left-scattered. 
Points that are simultaneously right-scattered and left-scattered 
are called isolated. If $\sigma(t)=t$, then $t$ is called right-dense, 
and if $\rho(t)=t$, then $t$ is left-dense. Points that 
are both right-dense and left-dense are said to be dense.
\end{definition}

\begin{definition}
The (forward) graininess function $\mu :\mathbb{T\rightarrow }\left[ 0,\infty \right)$
is given by $\mu (t)=\sigma(t)-t$, $t\in \mathbb{T}$. The set 
$\mathbb{T}^{\kappa}$ is defined as follows: if $\mathbb{T}$ has a left-scattered maximum 
$m$, then $\mathbb{T}^{\kappa}=\mathbb{T}-\left\{ m\right\} ;$ otherwise, 
$\mathbb{T}^{\kappa}=\mathbb{T}$.
\end{definition}

If $\mathbb{T}=\mathbb{R}$, then $\mu(t)\equiv0$; 
if $\mathbb{T}=\mathbb{Z}$, then we have $\mu (t)\equiv 1$.

\begin{definition}
Assume $f:\mathbb{T\rightarrow R}$ and fix $t\in \mathbb{T}^{\kappa}$.
Then the (delta) derivative $f^{\Delta }(t)\in \mathbb{R}$ at 
$t\in \mathbb{T}^{\kappa}$ is defined to be the number (provided it exists) 
with the property that given any $\epsilon >0$ there exists a neighborhood 
$U$ of $t$ such that
\begin{equation*}
\left\vert f(\sigma(t))-f(s)-f^{\Delta }(t)\left[\sigma(t)-s\right]
\right\vert \leq \epsilon \left\vert \sigma(t)-s\right\vert
\quad \forall s\in U.
\end{equation*}
\end{definition}

If $\mathbb{T=R}$, then $f^{\Delta }(t)=\frac{df(t)}{dt}$; 
if $\mathbb{T=Z}$, then $f^{\Delta }(t)=\Delta f(t)=f(t+1)-f(t)$.

\begin{theorem}
Assume $f,g:\mathbb{T\rightarrow R}$ are differentiable at $t\in \mathbb{T}^{\kappa}$. 
Then the product $fg:\mathbb{T\rightarrow R}$ is differentiable at $t$
with $\left( fg\right) ^{\Delta }(t)=f^{\Delta }(t)g(t)+f(\sigma(t))g^{\Delta}(t)$.
\end{theorem}

\begin{definition}
Function $f:\mathbb{T\rightarrow R}$ is said to be $rd$-continuous
if it is continuous at all right-dense points $t\in \mathbb{T}$ 
and its left-sided limits exist at all left-dense points $t\in \mathbb{T}$.
Then, we write: $f\in C_{rd}(\mathbb{T}$,$\mathbb{R})$.
\end{definition}

It turns out that every $rd$-continuous function has an anti-derivative
(see, e.g., Theorem 1.74 of \cite{BookTS:2003}).

\begin{definition}
\label{Pdef1}
Function $F:\mathbb{T\rightarrow R}$ is a delta anti-derivative of 
$f:\mathbb{T\rightarrow R}$ provided $F^{\Delta }(t)=f(t)$ for any $t\in
\mathbb{T}^{\kappa}$. In this case, one defines the $\Delta$-integral of $f$ by
\begin{equation*}
\int_{a}^{b}f(s)\Delta s:=F(b)-F(a)
\quad \text{for all} \quad a, b \in \mathbb{T}.
\end{equation*} 
\end{definition}

\begin{theorem}
Let $f,g$ be $rd$-continuous, $a,b,c\in \mathbb{T}$ 
and $\alpha ,\beta \in \mathbb{R}$. Then,
\begin{enumerate}
\item $\int_{a}^{b}\left[ \alpha f(t)+\beta g(t)\right] \Delta t
=\alpha \int_{a}^{b}f(t)\Delta t+\beta \int_{a}^{b}g(t)\Delta t$,

\item $\int_{a}^{b}f(t)\Delta t=-\int_{b}^{a}f(t)\Delta t$,

\item $\int_{a}^{b}f(t)\Delta t=\int_{a}^{c}f(t)\Delta t
+\int_{c}^{b}f(t)\Delta t$,

\item $\int_{a}^{b}f(t)g^{\Delta }(t)\Delta t=\left( fg\right)(b)
-\left(fg\right) (a)-\int_{a}^{b}f^{\Delta }(t)g\left(\sigma(t)\right)\Delta t$.
\end{enumerate}
\end{theorem}

We use the following result to prove our
generalized weighted Ostrowski 
inequality on time scales.

\begin{theorem}
\label{Pthm2}
If $f$ is $\Delta$-integrable on $\left[ a,b\right]$, then so is 
$\left\vert f\right\vert$, and
\begin{equation*}
\left\vert \int_{a}^{b}f(t)\Delta t\right\vert \leq \int_{a}^{b}\left\vert
f(t)\right\vert \Delta t.
\end{equation*}
\end{theorem}

We also make use of the $h_{k}$ polynomials, $k\in \mathbb{N}_{0}$. 
They are defined as follows: $h_{0}(t,s)=1$ for all $s,t\in \mathbb{T}$ and then, 
recursively, by $h_{k+1}\left( t,s\right) 
=\displaystyle \int_{s}^{t}h_{k}\left( \tau ,s\right) \Delta \tau$, 
$s,t\in \mathbb{T}$.


\section{Main results}
\label{sec3}

For the proof of our main results (Theorems~\ref{thm1} and \ref{thm2}), 
we use the following useful lemma.

\begin{lemma}[See \cite{Nwaeze}]
\label{Wlem1}
Let $\nu:[a, b]\rightarrow [0,\infty)$ be $rd$-continuous and positive 
and $w:[a, b]\rightarrow \mathbb{R}$ be delta-differentiable such that 
$w^{\Delta}(t )=\nu(t)$ on $[a, b]$. Moreover, suppose also that 
$a, b, s, t\in \mathbb{T}$, $a<b$, $f:[a, b]\rightarrow \mathbb{R}$ 
is delta-differentiable, and $\psi$ is a function of $[0, 1]$ into $[0, 1]$. 
Then the following equality holds:
\begin{multline*}
\left[\dfrac{1+\psi(1-\lambda)-\psi(\lambda)}{2}f(t) 
+ \dfrac{\psi(\lambda)f(a)+\left(1-\psi(1-\lambda)\right)f(b)}{2}\right]
\int_a^b \nu(t)\Delta t\\
=\int_a^b K(s,t)f^{\Delta}(s)\Delta s 
+ \int_a^b \nu(s)f(\sigma(s))\Delta s,
\end{multline*}
where $K(\cdot, \cdot)$ is given by \eqref{WE}.
\end{lemma}


\subsection{Generalized weighted Ostrowski inequality on time scales}

We now state and prove our first main result.

\begin{theorem}
\label{thm1}
Let $\nu:[a, b]\rightarrow [0,\infty)$ be $rd$-continuous and positive, 
and function $w:[a, b]\rightarrow \mathbb{R}$ be delta-differentiable such that 
$w^{\Delta}(t )=\nu(t)$ on $[a, b]$. Suppose also that $a, b, t, x\in \mathbb{T}$, 
$a<b$, $f:[a, b]\rightarrow \mathbb{R}$ is twice delta-differentiable, 
and $\psi$ is a function of $[0, 1]$ into $[0, 1]$. Then, 
the inequality
\begin{equation}
\label{IneqMR1}
\begin{split}
&\Bigg| \Phi^2(\lambda)f(x) - \frac{1}{\Big(\int_a^b \nu(t)\Delta t\Big)^2}
\Bigg(\int_a^b K(t, x)\Delta t\Bigg)\Bigg(\int_a^b \nu(s)f^{\Delta}(\sigma(s))\Delta s\Bigg)\\
&\quad +\dfrac{\psi(\lambda)f^{\Delta}(a)+\left(1-\psi(1-\lambda)\right)
f^{\Delta}(b)}{2\int_a^b \nu(t)\Delta t}\int_a^b K(t, x)\Delta t \\
&\quad - \frac{\Phi(\lambda)}{\int_a^b \nu(t)\Delta t}\int_a^b \nu(t)f(\sigma(t))\Delta t
+\dfrac{\psi(\lambda)f(a)+\left(1-\psi(1-\lambda)\right)f(b)}{2}\Phi(\lambda)\Bigg|\\
&\leq \frac{M}{\Big(\int_a^b \nu(t)\Delta t\Big)^2}\int_a^b\int_a^b |K(t, x)|| K(s, t)|\Delta s \Delta t
\end{split}
\end{equation}
holds for all $x\in [a, b]$ and $\lambda\in[0, 1]$, where
\begin{equation}
\label{eq:Phi}
\Phi(\lambda) =\dfrac{1+\psi(1-\lambda)-\psi(\lambda)}{2},
\end{equation}
$M=\sup\limits_{a<t<b}\Big|f^{\Delta\Delta}(t)\Big|<\infty$, 
and $K(\cdot, \cdot)$ is given by \eqref{WE}.
\end{theorem}

\begin{proof}
With $\Phi(\lambda)$ given by \eqref{eq:Phi}, it follows 
from Lemma~\ref{Wlem1} that
\begin{multline}
\label{Ineq1}
\Phi(\lambda)f(x)
=\frac{1}{\int_a^b \nu(t)\Delta t}\int_a^b K(t, x)f^{\Delta}(t)\Delta t 
+ \frac{1}{\int_a^b \nu(t)\Delta t}\int_a^b \nu(t)f(\sigma(t))\Delta t \\
-\dfrac{\psi(\lambda)f(a)+\left(1-\psi(1-\lambda)\right)f(b)}{2}.
\end{multline}
From \eqref{Ineq1} we get
\begin{multline}
\label{Ineq2}
\Phi^2(\lambda)f(x)
=\frac{\Phi(\lambda)}{\int_a^b \nu(t)\Delta t}
\int_a^b K(t, x)f^{\Delta}(t)\Delta t 
+ \frac{\Phi(\lambda)}{\int_a^b \nu(t)\Delta t}\int_a^b \nu(t)f(\sigma(t))\Delta t \\
-\dfrac{\psi(\lambda)f(a)+\left(1-\psi(1-\lambda)\right)f(b)}{2}\Phi(\lambda)
\end{multline}
and 
\begin{multline}
\label{Ineq3}
\Phi(\lambda)f^{\Delta}(t)
=\frac{1}{\int_a^b \nu(t)\Delta t}\int_a^b K(s, t)f^{\Delta\Delta}(s)\Delta s 
+ \frac{1}{\int_a^b \nu(t)\Delta t}\int_a^b \nu(s)f^{\Delta}(\sigma(s))\Delta s \\
-\dfrac{\psi(\lambda)f^{\Delta}(a)+\left(1-\psi(1-\lambda)\right)f^{\Delta}(b)}{2}.
\end{multline}
Substituting \eqref{Ineq3} into \eqref{Ineq2} results to
\begin{equation}
\label{Ineq4}
\begin{split}
\Phi^2&(\lambda)f(x)
=\frac{1}{\int_a^b \nu(t)\Delta t}\int_a^b K(t, x)\Bigg[\frac{1}{\int_a^b \nu(t)\Delta t}
\int_a^b K(s, t)f^{\Delta\Delta}(s)\Delta s \\
&\quad + \frac{1}{\int_a^b \nu(t)\Delta t}\int_a^b \nu(s)f^{\Delta}(\sigma(s))\Delta s
-\dfrac{\psi(\lambda)f^{\Delta}(a)+\left(1-\psi(1-\lambda)\right)f^{\Delta}(b)}{2}\Bigg]\Delta t\\
&\quad + \frac{\Phi(\lambda)}{\int_a^b \nu(t)\Delta t}\int_a^b \nu(t)f(\sigma(t))\Delta t
-\dfrac{\psi(\lambda)f(a)+\left(1-\psi(1-\lambda)\right)f(b)}{2}\Phi(\lambda)\\
&= \frac{1}{\Big(\int_a^b \nu(t)\Delta t\Big)^2}\Bigg(\int_a^b K(t, x)\Delta t\Bigg)\Bigg(
\int_a^b \nu(s)f^{\Delta}(\sigma(s))\Delta s\Bigg)\\
&\quad +\frac{1}{\Big(\int_a^b \nu(t)\Delta t\Big)^2}\int_a^b\int_a^b K(t, x) 
K(s, t)f^{\Delta\Delta}(s)\Delta s \Delta t\\
\end{split}
\end{equation}
\begin{equation*}
\begin{split}
&\quad -\dfrac{\psi(\lambda)f^{\Delta}(a)+\left(1-\psi(1-\lambda)\right)f^{\Delta}(b)}{2
\int_a^b \nu(t)\Delta t}\int_a^b K(t, x)\Delta t \\
&\quad + \frac{\Phi(\lambda)}{\int_a^b \nu(t)\Delta t}\int_a^b \nu(t)f(\sigma(t))\Delta t
-\dfrac{\psi(\lambda)f(a)+\left(1-\psi(1-\lambda)\right)f(b)}{2}\Phi(\lambda).
\end{split}
\end{equation*}
The desired inequality \eqref{IneqMR1} is obtained
by rearranging \eqref{Ineq4} and applying Theorem~\ref{Pthm2}.
\end{proof}

\begin{remark}
\label{rem1} 
Let $w(t)=t$. Then
\begin{multline*} 
\int_a^b |K(s, t)|\Delta s 
=h_2\left(a, a+\psi(\lambda)\frac{b-a}{2} \right) 
+ h_2\left(t, a+\psi(\lambda)\frac{b-a}{2} \right) \\
+ h_2\left(t, a+(1+\psi(1-\lambda))\frac{b-a}{2} \right)
+ h_2\left(b, a+(1+\psi(1-\lambda))\frac{b-a}{2} \right)
\end{multline*}
and 
\begin{multline*}
\int_a^b K(s, t)\Delta s 
= h_2\left(t, a+\psi(\lambda)\frac{b-a}{2} \right)
-h_2\left(a, a+\psi(\lambda)\frac{b-a}{2} \right)\\ 
+ h_2\left(b, a+(1+\psi(1-\lambda))\frac{b-a}{2} \right)
-h_2\left(t, a+(1+\psi(1-\lambda))\frac{b-a}{2} \right)
\end{multline*}
hold for all $\lambda\in [0, 1]$ such that $a+\psi(\lambda)\frac{b-a}{2}$ 
and $a+(1+\psi(1-\lambda))\frac{b-a}{2}$ are in $\mathbb{T}$ 
and $t\in \left[a+\psi(\lambda)\frac{b-a}{2}, 
a+(1+\psi(1-\lambda))\frac{b-a}{2}\right]$.
\end{remark}

\begin{corollary}
\label{corthm1}
Let $a, b, t, x\in \mathbb{T}$, $a<b$, and $f:[a, b]\rightarrow \mathbb{R}$ 
be twice delta-differentiable. Then, for all $x\in [a, b]$, the following inequality 
holds for all $\lambda\in [0, 1]$ such that $a+\lambda\frac{b-a}{2}$ 
and $a+(2-\lambda)\frac{b-a}{2}$ are in $\mathbb{T}$ and 
$t\in \left[a+\lambda\frac{b-a}{2}, a+(2-\lambda)\frac{b-a}{2}\right]$:
\begin{equation*}
\begin{split}
&\Bigg| (\lambda^2 - 2\lambda +1)f(x) - \frac{\int_a^b f^{\Delta}(\sigma(s))
\Delta s}{(b-a)^2}\Bigg[ h_2\left(x, a+\lambda\frac{b-a}{2} \right)
-h_2\left(a, a+\lambda\frac{b-a}{2} \right) \\
&\quad +  h_2\left(b, a+(2-\lambda)\frac{b-a}{2} \right)
-h_2\left(x, a+(2-\lambda)\frac{b-a}{2} \right)\Bigg]\\
&\quad +\lambda\dfrac{f^{\Delta}(a)+f^{\Delta}(b)}{2(b-a)}\Bigg[ 
h_2\left(x, a+\lambda\frac{b-a}{2} \right)-h_2\left(a, a+\lambda\frac{b-a}{2} \right)\\
&\quad +  h_2\left(b, a+(2-\lambda)\frac{b-a}{2} \right)
-h_2\left(x, a+(2-\lambda)\frac{b-a}{2} \right)\Bigg] \\
&\quad - \frac{1-\lambda}{b-a}\int_a^b f(\sigma(t))\Delta t
+\lambda(1-\lambda)\dfrac{f(a)+f(b)}{2}\Bigg|\\
&\leq \frac{M}{(b-a)^2}\int_a^b|K(t, x)|\Bigg[h_2\left(a, a+\lambda\frac{b-a}{2} \right) 
+ h_2\left(t, a+\lambda\frac{b-a}{2} \right) \\
&\quad + h_2\left(t, a+(2-\lambda)\frac{b-a}{2} \right)
+h_2\left(b, a+(2-\lambda)\frac{b-a}{2} \right)\Bigg]\Delta t,
\end{split}
\end{equation*}
where $M=\sup\limits_{a<t<b}\Big|f^{\Delta\Delta}(t)\Big|<\infty,$ and 
\begin{equation*}
K(t, x)=
\begin{cases}
t-\left(a+\lambda\frac{b-a}{2}\right), ~~~~t\in[a, x),\\
t-\left(a+(2-\lambda)\frac{b-a}{2}\right), ~~~~t\in[x, b].
\end{cases}
\end{equation*}
\end{corollary}

\begin{proof}
Let $w(t)=t$ and $\psi(\lambda)=\lambda$.
The result follows from Theorem~\ref{thm1} 
by using Remark~\ref{rem1}.
\end{proof}

As a particular case of Corollary~\ref{corthm1}, 
we obtain an extension of Theorem~\ref{DB} 
to an arbitrary time scale $\mathbb{T}$. 

\begin{corollary}
\label{corDB}
Let $a, b, t, x\in \mathbb{T}$, $a<b$, and $f:[a, b]\rightarrow \mathbb{R}$ 
be twice delta-differentiable. Then, the inequality 
\begin{multline*}
\Bigg| f(x) - \frac{\int_a^b f^{\Delta}(\sigma(s))\Delta s}{(b-a)^2}
\Big[ h_2(x, a)-h_2(x, b)\Big]- \frac{1}{b-a}\int_a^b f(\sigma(t))\Delta t\Bigg|\\
\leq \frac{M}{(b-a)^2}\int_a^b|K(t, x)|\Big[ h_2(t, a)+ h_2(t, b)\Big]\Delta t
\end{multline*}
holds for all $x\in [a, b]$,
where $M=\sup\limits_{a<t<b}\Big|f^{\Delta\Delta}(t)\Big|<\infty$ and 
$K(t, x)=
\begin{cases}
t-a, & t\in[a, x),\\
t-b, & t\in[x, b].
\end{cases}$
\end{corollary}

\begin{proof}
Choose $\lambda=0$ in Corollary~\ref{corthm1}. 
\end{proof}

\begin{remark}
\label{remDB}
Theorem~\ref{DB} is obtained from Corollary~\ref{corDB}
by choosing $\mathbb{T}=\mathbb{R}$.
\end{remark}

To the best of our knowledge, Theorem~\ref{thm1} is new even 
when we consider particular time scales as $\mathbb{T}=\mathbb{R}$,
$\mathbb{T}=\mathbb{Z}$ or $\mathbb{T}=q^{\mathbb{N}_0}$, $q>1$.

\begin{corollary}
\label{cor1}
Let $\nu:[a, b]\rightarrow [0,\infty)$ be continuous and positive 
and function $w:[a, b]\rightarrow \mathbb{R}$ be differentiable such that 
$w'(t)=\nu(t)$ on $[a, b]$. Suppose also that $a<b$, 
$f:[a, b]\rightarrow \mathbb{R}$ is twice differentiable, 
and $\psi$ is a function of $[0, 1]$ into $[0, 1]$. Then, 
the inequality
\begin{align*}
&\Bigg| \Phi^2(\lambda)f(x) - \frac{1}{\Big(
\int_a^b \nu(t)dt\Big)^2}\Bigg(\int_a^b K(t, x)dt\Bigg)\Bigg(\int_a^b \nu(s)f'(s)ds\Bigg)\\
&\quad +\dfrac{\psi(\lambda)f'(a)+\left(1-\psi(1-\lambda)\right)f'(b)}{2\int_a^b \nu(t)dt}\int_a^b K(t, x)dt\\
&- \frac{\Phi(\lambda)}{\int_a^b \nu(t)dt}\int_a^b \nu(t)f(t)dt+\dfrac{\psi(\lambda)f(a)+\left(1-\psi(1-\lambda)\right)f(b)}{2}\Phi(\lambda)\Bigg|\\
&\leq \frac{M}{\Big(\int_a^b \nu(t)dt\Big)^2}\int_a^b\int_a^b |K(t, x)|| K(s, t)|ds dt
\end{align*}
holds for all $x\in [a, b]$ and $\lambda\in[0, 1]$,
where $K(\cdot, \cdot)$ is given by \eqref{WE},
$\Phi(\lambda)$ by \eqref{eq:Phi}, and
$M=\sup\limits_{a<t<b}\Big|f''(t)\Big|<\infty$.
\end{corollary}

\begin{proof}
Choose $\mathbb{T}=\mathbb{R}$ in Theorem~\ref{thm1}.
\end{proof}

\begin{corollary}
\label{cor2}
Let $a, b \in \mathbb{Z}$, $a < b$, $\nu:\{a, \ldots, b\} \rightarrow [0,\infty)$ 
be positive and function $w:\{a, \ldots, b\}\rightarrow \mathbb{R}$ be such that 
$\Delta w(t)=\nu(t)$, $t =a, \ldots, b-1$. Consider also given functions 
$f:\{a, \ldots, b\}\rightarrow \mathbb{R}$ and $\psi : [0, 1]\rightarrow [0, 1]$. 
Then, the inequality
\begin{align*}
&\Bigg| \Phi^2(\lambda)f(x) - \frac{1}{\Big(\sum_{t=a}^{b-1} 
\nu(t)\Big)^2}\Bigg(\sum_{t=a}^{b-1} K(t, x)\Bigg)\Bigg(\sum_{s=a}^{b-1} \nu(s)(f(s+2)-f(s+1))\Bigg)\\
&\quad +\dfrac{\psi(\lambda)(f(a+1)-f(a))+\left(1-\psi(1-\lambda)\right)(f(b+1)-f(b))}{2
\sum_{t=a}^{b-1} \nu(t)}\sum_{t=a}^{b-1} K(t, x) \\
&\quad - \frac{\Phi(\lambda)}{\sum_{t=a}^{b-1} \nu(t)}\sum_{t=a}^{b-1}
\nu(t)f(t+1)+\dfrac{\psi(\lambda)f(a)+\left(1-\psi(1-\lambda)\right)f(b)}{2}\Phi(\lambda)\Bigg|\\
&\leq \frac{M}{\Big(\sum_{t=a}^{b-1} \nu(t)\Big)^2}
\sum_{t=a}^{b-1}\sum_{s=a}^{b-1} |K(t, x)|| K(s, t)|
\end{align*}
holds for all $x\in \{a, \ldots, b\}$ and $\lambda\in[0, 1]$,
where $K(\cdot, \cdot)$ is given by \eqref{WE},
$\Phi(\lambda)$ by \eqref{eq:Phi}, and
$M=\sup\limits_{a<t<b}\Big|f(t+1)-2f(t)+f(t-1)\Big|<\infty$.
\end{corollary}

\begin{proof}
Choose $\mathbb{T}=\mathbb{Z}$ in Theorem~\ref{thm1}.
\end{proof}

An interesting (quantum) calculus is obtained by choosing 
$\mathbb{T}=q^{\mathbb{N}_0}$ with $q>1$ \cite{MR1865777}. 
In this case, $\sigma(t)=qt$ and $f^{\Delta}(t)=D_{q}f(t):=\frac{f(qt)-f(t)}{(q-1)t}$. 
The corresponding integral is known in the literature 
as the $d_q t$ integral \cite{MR1865777}.

\begin{corollary}
\label{Ineqcor3}
Let $m, n \in \mathbb{N}$ with $m<n$,
$\nu:[q^m, q^n]\rightarrow [0,\infty)$ be positive, 
and function $w:[q^m, q^n]\rightarrow \mathbb{R}$ be such that 
$D_{q}w(t) = \nu(t)$ on $[q^m, q^n]$.
Consider also functions $f:[q^m, q^n]\rightarrow \mathbb{R}$  
and $\psi : [0, 1] \rightarrow [0, 1]$. Then, the inequality
\begin{align*}
&\Bigg| \Phi^2(\lambda)f(x) - \frac{1}{\Big(\int_{q^m}^{q^n} 
\nu(t)d_q t\Big)^2}\Bigg(\sum_{j=m}^{n-1} K(q^j, x)\Bigg)\Bigg(
\int_{q^m}^{q^n} \nu(s)\frac{f(q^2s)-f(qs)}{(q-1)qs}d_q s\Bigg)\\
&\quad +\dfrac{q^n\psi(\lambda)\big[f(q^{m+1})-f(q^m)\big]+q^m\left(1
-\psi(1-\lambda)\right)\big[f(q^{n+1})-f(q^n)\big]}{2q^{m+n}(q-1)
\int_{q^m}^{q^n} \nu(t)d_q t}\int_a^b K(t, x)\Delta t \\
&\quad - \frac{\Phi(\lambda)}{\int_{q^m}^{q^n} \nu(t)d_q t}
\int_{q^m}^{q^n} \nu(t)f(qt) d_qt +\dfrac{\psi(\lambda)f(q^m)
+\left(1-\psi(1-\lambda)\right)f(q^n)}{2}\Phi(\lambda)\Bigg|\\
&\leq \frac{M}{\Big(\int_{q^m}^{q^n} \nu(t)d_q t\Big)^2}
\sum_{j=m}^{n-1}\sum_{i=m}^{n-1} |K(q^j, x)||K(q^i, q^j)|
\end{align*}
holds for all $x\in [q^m, q^n]$ and $\lambda\in[0, 1]$, where 
$$
M=\sup\limits_{q^m<t<q^n}
\Big|\frac{f(q^2t)-(q+1)f(qt)+qf(t)}{q(q-1)^2t^2}\Big|<\infty,
$$
\begin{equation}
\label{eq:K:quantum}
K(q^j, x)=
\begin{cases}
w(q^j)-\left(w(q^m)+\psi(\lambda)\frac{w(q^n)-w(q^m)}{2}\right),
\quad q^j\in[q^m, x),\\
w(q^j)-\left(w(q^m)+(1+\psi(1-\lambda))\frac{w(q^n)-w(q^m)}{2}\right), 
\quad q^j\in[x, q^n],
\end{cases}
\end{equation}
and $\Phi(\lambda)$ is given by \eqref{eq:Phi}.
\end{corollary}

\begin{proof}
Choose $\mathbb{T}=q^{\mathbb{N}_0}$, $q>1$, 
and $a=q^m$ and $b=q^n$, $m<n$, in Theorem~\ref{thm1}.
\end{proof}


\subsection{Generalized weighted Ostrowski--Gr\"uss inequality on time scales}

Follows the second main result of our paper.

\begin{theorem}
\label{thm2}
Let $\nu:[a, b]\rightarrow [0,\infty)$ be $rd$-continuous and positive, 
and function $w:[a, b]\rightarrow \mathbb{R}$ be delta-differentiable such that 
$w^{\Delta}(t )=\nu(t)$ on $[a, b]$. Suppose also that $a, b, t, x\in \mathbb{T}$, 
$a<b$, $f:[a, b]\rightarrow \mathbb{R}$ is delta-differentiable 
and $\psi$ is a function of $[0, 1]$ into $[0, 1]$. 
Then, the inequality
\begin{equation}
\label{IneqMR2}
\begin{split}
&\Bigg|\left[\dfrac{1+\psi(1-\lambda)-\psi(\lambda)}{2(b-a)}f(x) 
+ \dfrac{\psi(\lambda)f(a)+\left(1-\psi(1-\lambda)\right)f(b)}{2(b-a)}\right]
\int_a^b \nu(t)\Delta t\\
& \quad - \frac{1}{b-a}\int_a^b \nu(t)f(\sigma(t))\Delta t 
-\Bigg(\dfrac{f(b)-f(a)}{(b-a)^2}\int_a^b K(t, x)\Delta t \Bigg)\Bigg|\\
&\leq \Bigg[\dfrac{1}{b-a}\int_a^b K^2(t, x)\Delta t 
- \Bigg(\dfrac{1}{b-a}\int_a^b K(t, x)\Delta t \Bigg)^2\Bigg]^\frac{1}{2}\\
&\quad \times\Bigg[\dfrac{1}{b-a}\int_a^b \Big(f^{\Delta}(t)\Big)^2\Delta t 
- \Bigg(\dfrac{1}{b-a}\int_a^b f^{\Delta}(t)\Delta t \Bigg)^2\Bigg]^\frac{1}{2}
\end{split}
\end{equation}
holds for all $ x\in [a, b]$ and $\lambda\in[0, 1]$,  
where $K(\cdot,\cdot)$ is defined as in \eqref{WE}.
\end{theorem}

\begin{proof}
We start by making the following computations: 
\begin{equation*}
\begin{split}
&\int_a^b\int_a^b \Big(K(t, x) - K(s, x)\Big)\Big(f^{\Delta}(t) 
- f^{\Delta}(s)\Big)\Delta t\Delta s\\
&= (b-a)\int_a^b K(t, x)f^{\Delta}(t)\Delta t 
- \Bigg(\int_a^b K(t, x)\Delta t \Bigg)\Bigg(\int_a^bf^{\Delta}(s)\Delta s \Bigg)\\
&\quad - \Bigg(\int_a^b K(s, x)\Delta s \Bigg)\Bigg(\int_a^bf^{\Delta}(t)\Delta t \Bigg)
+ (b-a)\int_a^b K(s, x)f^{\Delta}(s)\Delta s\\
&= 2(b-a)\int_a^b K(t, x)f^{\Delta}(t)\Delta t 
- 2\Bigg(\int_a^b K(t, x)\Delta t \Bigg)\Bigg(\int_a^bf^{\Delta}(s)\Delta s \Bigg).
\end{split}
\end{equation*}
This implies that 
\begin{multline}
\label{Ineq5}
\dfrac{1}{b-a}\int_a^b K(t, x)f^{\Delta}(t)\Delta t 
- \Bigg(\dfrac{1}{b-a}\int_a^b K(t, x)\Delta t \Bigg)\Bigg(
\dfrac{1}{b-a}\int_a^bf^{\Delta}(s)\Delta s \Bigg)\\
=\dfrac{1}{2(b-a)^2}\int_a^b\int_a^b \Big(K(t, x) 
- K(s, x)\Big)\Big(f^{\Delta}(t) - f^{\Delta}(s)\Big)\Delta t\Delta s.
\end{multline}
Following the same process, one gets
\begin{multline}
\label{Ineq6}
\dfrac{1}{b-a}\int_a^b K^2(t, x)\Delta t 
- \Bigg(\dfrac{1}{b-a}\int_a^b K(t, x)\Delta t \Bigg)^2\\
=\dfrac{1}{2(b-a)^2}\int_a^b\int_a^b \Big(K(t, x) 
- K(s, x)\Big)^2\Delta t\Delta s
\end{multline}
and 
\begin{multline}
\label{Ineq7}
\dfrac{1}{b-a}\int_a^b \Big(f^{\Delta}(t)\Big)^2\Delta t
- \Bigg(\dfrac{1}{b-a}\int_a^b f^{\Delta}(t)\Delta t \Bigg)^2\\
=\dfrac{1}{2(b-a)^2}\int_a^b\int_a^b \Big(f^{\Delta}(t) 
- f^{\Delta}(s)\Big)^2\Delta t\Delta s.
\end{multline}
From Lemma~\ref{Wlem1}, we also get that
\begin{multline}
\label{Ineq8}
\left[\dfrac{1+\psi(1-\lambda)-\psi(\lambda)}{2}f(x) 
+ \dfrac{\psi(\lambda)f(a)+\left(1-\psi(1-\lambda)\right)f(b)}{2}\right]
\int_a^b \nu(t)\Delta t\\
- \int_a^b \nu(t)f(\sigma(t))\Delta t 
= \int_a^b K(t,x)f^{\Delta}(t)\Delta t.
\end{multline}
Using the Cauchy--Schwarz inequality on time scales 
(see \cite{MR1859660}), we get
\begin{equation}
\label{Ineq9}
\begin{split}
\Bigg|&\dfrac{1}{2(b-a)^2}\int_a^b\int_a^b \Big(K(t, x) 
- K(s, x)\Big)\Big(f^{\Delta}(t) - f^{\Delta}(s)\Big)\Delta t\Delta s\Bigg|\\
&\leq \Bigg[\dfrac{1}{2(b-a)^2}\int_a^b\int_a^b \Big(K(t, x) 
- K(s, x)\Big)^2\Delta t\Delta s\Bigg]^\frac{1}{2}\\
&\quad \times\Bigg[\dfrac{1}{2(b-a)^2}\int_a^b\int_a^b \Big(f^{\Delta}(t) 
- f^{\Delta}(s)\Big)^2\Delta t\Delta s\Bigg]^\frac{1}{2}.
\end{split}
\end{equation}
Inequality \eqref{IneqMR2} is achieved by applying 
\eqref{Ineq5}--\eqref{Ineq8} and Definition~\ref{Pdef1} to \eqref{Ineq9}.
\end{proof}

\begin{corollary}
\label{corMR2}
Let $\mathbb{T}$ be a time scale with 
$a, b\in \mathbb{T}$, $a<b$, and $f:[a, b]\rightarrow \mathbb{R}$ 
be delta-differentiable. Then the inequality 
\begin{align*}
&\Bigg|\left(\dfrac{1-\lambda}{b-a}f(x) +\lambda\dfrac{f(a)+f(b)}{2(b-a)}\right)
\int_a^b \big(\sigma(t)+t\big)\Delta t\\
& \qquad - \frac{1}{b-a}\int_a^b \big(\sigma(t)+t\big)f(\sigma(t))\Delta t 
-  \Bigg(\dfrac{f(b)-f(a)}{(b-a)^2}\int_a^b K(t, x)\Delta t \Bigg)\Bigg|\\
&\leq \Bigg[\dfrac{1}{b-a}\int_a^b K^2(t, x)\Delta t - \Bigg(\dfrac{1}{b-a}
\int_a^b K(t, x)\Delta t \Bigg)^2\Bigg]^\frac{1}{2}\\
&\qquad \times\Bigg[\dfrac{1}{b-a}\int_a^b \Big(f^{\Delta}(t)\Big)^2\Delta t 
- \Bigg(\dfrac{1}{b-a}\int_a^b f^{\Delta}(t)\Delta t \Bigg)^2\Bigg]^\frac{1}{2}
\end{align*}
holds for all $x\in [a, b]$ and $\lambda\in[0, 1]$,
where $K(\cdot,\cdot)$ is defined by
\begin{equation}
\label{eq:K:aftr}
K(t, x)=
\begin{cases}
t^2-a^2 -\frac{\lambda}{2}(b^2-a^2), \quad t\in[a, x),\\
t^2-a^2-\frac{2-\lambda}{2}(b^2-a^2), \quad t\in[x, b].
\end{cases}
\end{equation}
\end{corollary}

\begin{proof}
Let $\psi(\lambda)=\lambda$ and $w(t)=t^2+c$, $c\in\mathbb{R}$,
in Theorem~\ref{thm2}. The result follows because $\nu(t)=\sigma(t)+t$ 
for $t\in [a, b]$. 
\end{proof}

\begin{corollary} 
Let $\mathbb{T}$ be a time scale with 
$a, b\in \mathbb{T}$, $a<b$, and $f:[a, b]\rightarrow \mathbb{R}$ 
be delta-differentiable. Then the inequality 
\begin{equation*}
\begin{split}
&\Bigg|\dfrac{1}{b-a}f(x)\int_a^b \big(\sigma(t)+t\big)\Delta t 
- \frac{1}{b-a}\int_a^b \big(\sigma(t)+t\big)f(\sigma(t))\Delta t \\
&\qquad -  \Bigg(\dfrac{f(b)-f(a)}{(b-a)^2}\int_a^b K(t, x)\Delta t \Bigg)\Bigg|\\
&\leq \Bigg[\dfrac{1}{b-a}\int_a^b K^2(t, x)\Delta t - \Bigg(\dfrac{1}{b-a}
\int_a^b K(t, x)\Delta t \Bigg)^2\Bigg]^\frac{1}{2}\\
&\qquad \times\Bigg[\dfrac{1}{b-a}\int_a^b \Big(f^{\Delta}(t)\Big)^2\Delta t 
- \Bigg(\dfrac{1}{b-a}\int_a^b f^{\Delta}(t)\Delta t \Bigg)^2\Bigg]^\frac{1}{2}
\end{split}
\end{equation*}
holds for all $x\in [a, b]$, where $K(\cdot,\cdot)$ is defined by
\eqref{eq:K:aftr}.
\end{corollary}

\begin{proof}
Choose $\lambda=0$ in the inequality of Corollary~\ref{corMR2}.
\end{proof}

Our Theorem~\ref{thm2} is new even for very standard time scales.

\begin{corollary}
\label{cor3}
If $w, f : [a, b]\rightarrow \mathbb{R}$ are differentiable,
$\psi : [0, 1]\rightarrow [0, 1]$, then
\begin{equation*}
\begin{split}
&\Bigg|\left[\dfrac{1+\psi(1-\lambda)-\psi(\lambda)}{2(b-a)}f(x) 
+ \dfrac{\psi(\lambda)f(a)+\left(1-\psi(1-\lambda)\right)f(b)}{2(b-a)}\right]\int_a^b w'(t)dt\\
& \qquad - \frac{1}{b-a}\int_a^b w'(t)f(t)dt 
-\Bigg(\dfrac{f(b)-f(a)}{(b-a)^2}\int_a^b K(t, x)dt \Bigg)\Bigg|\\
&\leq \Bigg[\dfrac{1}{b-a}\int_a^b K^2(t, x)dt - \Bigg(\dfrac{1}{b-a}
\int_a^b K(t, x)dt \Bigg)^2\Bigg]^\frac{1}{2}\\
&\qquad \times\Bigg[\dfrac{1}{b-a}\int_a^b \Big(f'(t)\Big)^2dt 
- \Bigg(\dfrac{1}{b-a}\int_a^b f'(t)dt \Bigg)^2\Bigg]^\frac{1}{2}
\end{split}
\end{equation*}
holds for all $ x\in [a, b]$ and $\lambda\in[0, 1]$,  
where $K(\cdot,\cdot)$ is given by \eqref{WE}.
\end{corollary}

\begin{proof}
Choose $\mathbb{T}=\mathbb{R}$ in Theorem~\ref{thm2}.
\end{proof}

\begin{corollary}
\label{cor4}
If $w, f :\{a,a+1, \ldots, b-1, b\} \rightarrow \mathbb{R}$, 
$\psi : [0, 1]\rightarrow [0, 1]$, then
\begin{equation*}
\begin{split}
&\Bigg|\left[\dfrac{1+\psi(1-\lambda)
-\psi(\lambda)}{2(b-a)}f(x) + \dfrac{\psi(\lambda)f(a)
+\left(1-\psi(1-\lambda)\right)f(b)}{2(b-a)}\right]\sum_{t=a}^{b-1} \Delta w(t)\\
& \qquad - \frac{1}{b-a}\sum_{t=a}^{b-1}\Delta w(t)f(t+1) - \Bigg(\dfrac{f(b)-f(a)}{(b-a)^2}
\sum_{t=a}^{b-1} K(t, x)\Bigg)\Bigg|\\
&\leq \Bigg[\dfrac{1}{b-a}\sum_{t=a}^{b-1} K^2(t, x) 
- \Bigg(\dfrac{1}{b-a}\sum_{t=a}^{b-1} K(t, x)\Bigg)^2\Bigg]^\frac{1}{2}\\
&\qquad \times\Bigg[\dfrac{1}{b-a}\sum_{t=a}^{b-1} \Big(\Delta f(t)\Big)^2 
- \Bigg(\dfrac{1}{b-a}\sum_{t=a}^{b-1} \Delta f(t) \Bigg)^2\Bigg]^\frac{1}{2}
\end{split}
\end{equation*}
holds for all $ x\in \{a,a+1, \ldots, b-1, b\}$ and $\lambda\in[0, 1]$,  
where $K(\cdot,\cdot)$ is given by \eqref{WE} and $\Delta$ denotes 
the forward difference operator, that is, $\Delta \xi(t) = \xi(t+1)-\xi(t)$.
\end{corollary}

\begin{proof}
Choose $\mathbb{T}=\mathbb{Z}$ in Theorem~\ref{thm2}.
\end{proof}

\begin{corollary}
\label{cor5}
Let $m, n \in \mathbb{N}$ with $m<n$,
$w, f:[q^m, q^n]\rightarrow \mathbb{R}$ 
and $\psi : [0, 1] \rightarrow [0, 1]$. Then, the inequality
\begin{equation*}
\begin{split}
&\Bigg|\left[\dfrac{1+\psi(1-\lambda)-\psi(\lambda)}{2(q^n - q^m)}f(x) 
+ \dfrac{\psi(\lambda)f(q^m)+\left(1-\psi(1-\lambda)\right)f(q^n)}{2(q^n - q^m)}\right]
\int_{q^m}^{q^n} D_{q}w(t) d_qt\\
& \quad - \frac{1}{q^n - q^m}\int_{q^m}^{q^n} D_{q}w(t)f(qt) d_qt 
-\Bigg(\dfrac{f(b)-f(a)}{(q^n - q^m)^2}\sum_{j=m}^{n-1}K(q^j,x)\Bigg)\Bigg|\\
\end{split}
\end{equation*}
\begin{equation*}
\begin{split}
&\leq \Bigg[\dfrac{1}{q^n-q^m}\sum_{j=m}^{n-1}K^2(q^j,x) 
- \Bigg(\dfrac{1}{q^n-q^m}\sum_{j=m}^{n-1}K(q^j,x) \Bigg)^2\Bigg]^\frac{1}{2}\\
&\quad \times\Bigg[\dfrac{1}{q^n-q^m}\sum_{j=m}^{n-1}
\Bigg(\dfrac{f(q^{j+1}) - f(q^j)}{(q-1)q^j}\Bigg)^2 
- \Bigg(\dfrac{1}{q^n-q^m}\sum_{j=m}^{n-1}\dfrac{f(q^{j+1}) 
- f(q^j)}{(q-1)q^j}\Bigg)^2\Bigg]^\frac{1}{2}
\end{split}
\end{equation*}
holds for all $x\in [q^m, q^n]$ and 
$\lambda \in [0,1]$ with $K(q^j, x)$ 
given by \eqref{eq:K:quantum}.
\end{corollary}

\begin{proof}
Let $\mathbb{T}=q^{\mathbb{N}_0}$ with $q>1$, $a=q^m$ and $b=q^n$, $m<n$. 
Then the results is a direct consequence of Theorem~\ref{thm2}.
\end{proof}

Some other special cases of our Theorem~\ref{thm2}
can be found in \cite{MR3311686}.


\section*{Acknowledgement}

The authors would like to thank an anonymous 
Referee for several comments and suggestions, 
which were very useful to improve the paper. 


\Refs

\bib{MR3307947}{book}{
	author={Agarwal, Ravi},
	author={O'Regan, Donal},
	author={Saker, Samir},
	title={Dynamic inequalities on time scales},
	publisher={Springer, Cham},
	date={2014},
	pages={x+256},
	isbn={978-3-319-11001-1},
	isbn={978-3-319-11002-8},
	doi={10.1007/978-3-319-11002-8},
}

\bib{MR1859660}{article}{
	author={Agarwal, R.},
	author={Bohner, M.},
	author={Peterson, A.},
	title={Inequalities on time scales: a survey},
	journal={Math. Inequal. Appl.},
	volume={4},
	date={2001},
	number={4},
	pages={535--557},
	issn={1331-4343},
	doi={10.7153/mia-04-48},
}

\bib{MR3445007}{book}{
	author={Anastassiou, George A.},
	title={Frontiers in time scales and inequalities},
	series={Series on Concrete and Applicable Mathematics},
	volume={17},
	publisher={World Scientific Publishing Co. Pte. Ltd., Hackensack, NJ},
	date={2016},
	pages={ix+278},
	isbn={978-981-4704-43-4},
}

\bib{MR3311686}{article}{
	author={Awan, Khalid Mahmood},
	author={Pe\v cari\'c, Josip},
	author={Penava, Mihaela Ribi\v ci\'c},
	title={Companion inequalities to Ostrowski-Gr\"uss type inequality and
		applications},
	journal={Turkish J. Math.},
	volume={39},
	date={2015},
	number={2},
	pages={228--234},
	issn={1300-0098},
	doi={10.3906/mat-1404-27},
}

\bib{MR3526108}{book}{
	author={Bari{\'c}, Josipa},
	author={Bibi, Rabia},
	author={Bohner, Martin},
	author={Nosheen, Ammara},
	author={Pe{\v{c}}ari{\'c}, Josip},
	title={Jensen inequalities on time scales},
	series={Monographs in Inequalities},
	volume={9},
	note={Theory and applications},
	publisher={ELEMENT, Zagreb},
	date={2015},
	pages={ix+284},
	isbn={978-953-197-597-1},
}

\bib{BM2}{article}{
	author={Bohner, Martin},
	author={Matthews, Thomas},
	title={The Gr\"uss inequality on time scales},
	journal={Commun. Math. Anal.},
	volume={3},
	date={2007},
	number={1},
	pages={1--8},
	issn={1938-9787},
}

\bib{BM1}{article}{
	author={Bohner, Martin},
	author={Matthews, Thomas},
	title={Ostrowski inequalities on time scales},
	journal={JIPAM. J. Inequal. Pure Appl. Math.},
	volume={9},
	date={2008},
	number={1},
	pages={Article 6, 8},
	issn={1443-5756},
}

\bib{BookTS:2001}{book}{
	author={Bohner, Martin},
	author={Peterson, Allan},
	title={Dynamic equations on time scales},
	note={An introduction with applications},
	publisher={Birkh\"auser Boston, Inc., Boston, MA},
	date={2001},
	pages={x+358},
	isbn={0-8176-4225-0},
	doi={10.1007/978-1-4612-0201-1},
}

\bib{BookTS:2003}{collection}{
	author={Bohner, Martin},
	author={Peterson, Allan},	
	title={Advances in dynamic equations on time scales},
	publisher={Birkh\"auser Boston, Inc., Boston, MA},
	date={2003},
	pages={xii+348},
	isbn={0-8176-4293-5},
	doi={10.1007/978-0-8176-8230-9},
}

\bib{DB}{article}{
	author={Dragomir, S. S.},
	author={Barnett, N. S.},
	title={An Ostrowski type inequality for mappings whose second derivatives
		are bounded and applications},
	journal={J. Indian Math. Soc. (N.S.)},
	volume={66},
	date={1999},
	number={1-4},
	pages={237--245},
	issn={0019-5839},
}

\bib{Hilger}{article}{
	author={Hilger, Stefan},
	title={Analysis on measure chains---a unified approach to continuous and
		discrete calculus},
	journal={Results Math.},
	volume={18},
	date={1990},
	number={1-2},
	pages={18--56},
	issn={0378-6218},
	doi={10.1007/BF03323153},
}

\bib{MR1865777}{book}{
	author={Kac, Victor},
	author={Cheung, Pokman},
	title={Quantum calculus},
	series={Universitext},
	publisher={Springer-Verlag, New York},
	date={2002},
	pages={x+112},
	isbn={0-387-95341-8},
	doi={10.1007/978-1-4613-0071-7},
}

\bib{LJ}{article}{
	author={Liu, Wen-jun},
	title={Some weighted integral inequalities with a parameter and
		applications},
	journal={Acta Appl. Math.},
	volume={109},
	date={2010},
	number={2},
	pages={389--400},
	issn={0167-8019},
	doi={10.1007/s10440-008-9323-2},
}

\bib{Liu2}{article}{
	author={Liu, Wenjun},
	author={Ng{\^o}, Qu{\~{o}}c Anh},
	author={Chen, Wenbing},
	title={Ostrowski type inequalities on time scales for double integrals},
	journal={Acta Appl. Math.},
	volume={110},
	date={2010},
	number={1},
	pages={477--497},
	issn={0167-8019},
	doi={10.1007/s10440-009-9456-y},
}

\bib{Liu1}{article}{
	author={Liu, Wenjun},
	author={Ng{\^o}, Qu{\~{o}}c Anh},
	author={Chen, Wenbing},
	title={On new Ostrowski type inequalities for double integrals on time
		scales},
	journal={Dynam. Systems Appl.},
	volume={19},
	date={2010},
	number={1},
	pages={189--198},
	issn={1056-2176},
}

\bib{Liu3}{article}{
	author={Liu, Wenjun},
	author={Tuna, Adnan},
	author={Jiang, Yong},
	title={On weighted Ostrowski type, trapezoid type, Gr\"uss type and
		Ostrowski-Gr\"uss like inequalities on time scales},
	journal={Appl. Anal.},
	volume={93},
	date={2014},
	number={3},
	pages={551--571},
	issn={0003-6811},
	doi={10.1080/00036811.2013.786045},
}

\bib{Liu4}{article}{
	author={Liu, Wenjun},
	author={Tuna, Adnan},
	author={Jiang, Yong},
	title={New weighted Ostrowski and Ostrowski-Gr\"uss type inequalities on
		time scales},
	journal={An. \c Stiin\c t. Univ. Al. I. Cuza Ia\c si. Mat. (N.S.)},
	volume={60},
	date={2014},
	number={1},
	pages={57--76},
	issn={1221-8421},
	doi={10.2478/aicu-2013-0002},
}

\bib{MR2816120}{article}{
	author={Mozyrska, Dorota},
	author={Paw{\l}uszewicz, Ewa},
	author={Torres, Delfim F. M.},
	title={Inequalities and majorisations for the Riemann-Stieltjes integral
		on time scales},
	journal={Math. Inequal. Appl.},
	volume={14},
	date={2011},
	number={2},
	pages={281--293},
	issn={1331-4343},
	doi={10.7153/mia-14-23},
	arXiv={1003.4772},
}

\bib{Nwaeze}{article}{
	author={Nwaeze, Eze R.},
	title={A new weighted Ostrowski type inequality on arbitrary time scale},
	journal={J. King Saud Univ. Sci.},
	volume={29},
	date={2017},
	number={2},
	pages={230--234},
	doi={10.1016/j.jksus.2016.09.006},
}

\bib{Ozkan1}{article}{
	author={{\"O}zkan, Umut Mutlu},
	author={Yildirim, H{\"u}seyin},
	title={Gr\"uss type inequalities for double integrals on time scales},
	journal={Comput. Math. Appl.},
	volume={57},
	date={2009},
	number={3},
	pages={436--444},
	issn={0898-1221},
	doi={10.1016/j.camwa.2008.11.006},
}

\bib{Ozkan2}{article}{
	author={{\"O}zkan, Umut Mutlu},
	author={Yildirim, H{\"u}seyin},
	title={Ostrowski type inequality for double integrals on time scales},
	journal={Acta Appl. Math.},
	volume={110},
	date={2010},
	number={1},
	pages={283--288},
	issn={0167-8019},
	doi={10.1007/s10440-008-9407-z},
}

\bib{XuFang}{article}{
	author={Xu, Gaopeng},  
	author={Fang, Zhong Bo}, 
	title={A new Ostrowski type inequality on time scales},
	journal={J. Math. Inequal.},
	volume={10},
	date={2016},
	number={3},
	pages={751--760},
}

\endRefs


\end{document}